\documentclass{amsart}

\usepackage[T1]{fontenc}
\usepackage{amssymb}
\usepackage{enumitem}
\usepackage{mathrsfs}
\usepackage[colorlinks, linkcolor=blue, citecolor=blue, urlcolor=blue]{hyperref}
\usepackage{tikz}

\newtheorem{theorem}{Theorem}[section]
\newtheorem{corollary}[theorem]{Corollary}
\newtheorem{lemma}[theorem]{Lemma}

\theoremstyle{definition}
\newtheorem{definition}[theorem]{Definition}

\newcommand{\rstr}{{\upharpoonright}}

\DeclareMathOperator{\scrP}{\mathscr{P}}
\DeclareMathOperator{\dom}{dom}
\DeclareMathOperator{\ran}{ran}
\DeclareMathOperator{\rank}{rank}

\title{A Note on Surjective Cardinals}

\author{Jiaheng Jin}
\address{School of Philosophy\\
Wuhan University\\
No.~299 Bayi Road\\
Wuhan 430072\\
Hubei Province\\
People's Republic of China}
\email{jin\_jiaheng@outlook.com}

\author{Guozhen Shen}
\address{Department of Philosophy (Zhuhai)\\
Sun Yat-sen University\\
No.~2 Daxue Road\\
Zhuhai 519082\\
Guangdong Province\\
People's Republic of China}
\email{shen\_guozhen@outlook.com}



\begin{document}

\begin{abstract}
For cardinals $\mathfrak{a}$ and $\mathfrak{b}$, we write $\mathfrak{a}=^\ast\mathfrak{b}$ if
there are sets $A$ and $B$ of cardinalities $\mathfrak{a}$ and $\mathfrak{b}$, respectively,
such that there are partial surjections from $A$ onto $B$ and from $B$ onto $A$.
$=^\ast$-equivalence classes are called surjective cardinals. In this article,
we show that $\mathsf{ZF}+\mathsf{DC}_\kappa$, where $\kappa$ is a fixed aleph,
cannot prove that surjective cardinals form a cardinal algebra,
which gives a negative solution to a question proposed by Truss [J. Truss, Ann. Pure Appl. Logic 27, 165--207 (1984)].
Nevertheless, we show that surjective cardinals form a ``surjective cardinal algebra'',
whose postulates are almost the same as those of a cardinal algebra,
except that the refinement postulate is replaced by the finite refinement postulate.
This yields a smoother proof of the cancellation law for surjective cardinals,
which states that $m\cdot\mathfrak{a}=^\ast m\cdot\mathfrak{b}$ implies $\mathfrak{a}=^\ast\mathfrak{b}$
for all cardinals $\mathfrak{a},\mathfrak{b}$ and all nonzero natural numbers $m$.
\end{abstract}

\maketitle

\section{Introduction and definitions}

The notion of a cardinal algebra, initiated by Tarski in his masterful book~\cite{Tarski1949b},
provides a common generalization for a number of important mathematical structures:
nonnegative real numbers under addition, sets of nonnegative measurable functions and countably additive measures
on a measurable space under pointwise summation, sets of Borel isomorphism types under Borel sum, and so forth.

A \emph{cardinal algebra} is an algebraic system $\langle A,{+},{\sum}\rangle$ which satisfies the following postulates I--VII.
\begin{description}
  \item[I (Finite closure postulate)] If $a,b\in A$, then $a+b\in A$.
  \item[II (Infinite closure postulate)] If $a_n\in A$ for all $n\in\omega$, then $\sum_{n\in\omega}a_n\in A$.
  \item[III (Associative postulate)] If $a_n\in A$ for all $n\in\omega$, then
  \[
  \sum_{n\in\omega}a_n=a_0+\sum_{n\in\omega}a_{n+1}.
  \]
  \item[IV (Commutative-associative postulate)] If $a_n,b_n\in A$ for all $n\in\omega$, then
  \[
  \sum_{n\in\omega}(a_n+b_n)=\sum_{n\in\omega}a_n+\sum_{n\in\omega}b_n.
  \]
  \item[V (Postulate of the zero element)] There is an element $0\in A$ such that $a+0=0+a=a$ for all $a\in A$.
  \item[VI (Refinement postulate)] If $a,b,c_n\in A$ for all $n\in\omega$ and $a+b=\sum_{n\in\omega}c_n$,
  then there are elements $a_n,b_n\in A$ for each $n\in\omega$ such that
  \[
  a=\sum_{n\in\omega}a_n,\quad b=\sum_{n\in\omega}b_n,\quad\text{and}\quad c_n=a_n+b_n\text{ for all }n\in\omega.
  \]
  \item[VII (Remainder postulate)]If $a_n,b_n\in A$ and $a_n=a_{n+1}+b_n$ for all $n\in\omega$,
  then there is an element $c\in A$ such that
  \[
  a_m=c+\sum_{n\in\omega}b_{m+n}\text{ for all }m\in\omega.
  \]
\end{description}
It is clear that, assuming the countable axiom of choice $\mathsf{AC}_\omega$, cardinals form a cardinal algebra.

Let $\kappa$ be an aleph. Recall the principle of $\kappa$-dependent choices $\mathsf{DC}_\kappa$.
\begin{description}
  \item[$\mathsf{DC}_\kappa$] Let $S$ be a set and let $R$ be a binary relation such that
  for each $\alpha<\kappa$ and each $\alpha$-sequence $s=\langle x_\xi\rangle_{\xi<\alpha}$ of elements of $S$
  there is $y\in S$ such that $sRy$. Then there is a function $f:\kappa\to S$ such that $(f\rstr\alpha)Rf(\alpha)$ for every $\alpha<\kappa$.
\end{description}
It is shown in \cite[Corollary~2.34]{Tarski1949b} that, assuming $\mathsf{DC}_\omega$,
in any cardinal algebra $\langle A,{+},{\sum}\rangle$, if $a,b\in A$ and $m\in\omega\setminus\{0\}$,
then $m\cdot a=m\cdot b$ implies $a=b$. This yields a choice-free proof of the celebrated Bernstein division theorem,
which states that $m\cdot\mathfrak{a}=m\cdot\mathfrak{b}$ implies $\mathfrak{a}=\mathfrak{b}$
for all cardinals $\mathfrak{a},\mathfrak{b}$ and all nonzero natural numbers $m$.
(Although $\mathsf{DC}_\omega$ is needed in the \emph{algebraic} proof of \cite[Corollary~2.34]{Tarski1949b},
as remarked by Tarski~\cite[pp.~240--242]{Tarski1949b}, in cardinal \emph{arithmetic},
we can dispense with the use of $\mathsf{DC}_\omega$ in the proof of the Bernstein division theorem.)

``Weak cardinal algebras'' were introduced by Truss~\cite{Truss1973} in an attempt to derive
as many properties of cardinal algebras as possible using only finitary addition $+$.
The infinitary defining postulates of a cardinal algebra were replaced by the following
``finite refinement'' and ``approximate cancellation'' postulates.
\begin{description}
  \item[VI' (Finite refinement postulate)] If $a_1,a_2,b_1,b_2\in A$ and $a_1+a_2=b_1+b_2$,
  then there are elements $c_1,c_2,c_3,c_4\in A$ such that $a_1=c_1+c_2$, $a_2=c_3+c_4$,
  $b_1=c_1+c_3$, and $b_2=c_2+c_4$.
  \item[VIII (Approximate cancellation postulate)]If $a,b,c\in A$ and $a+c=b+c$,
  then there are elements $a',b',d\in A$ such that $a=a'+d$, $b=b'+d$, and $c=a'+c=b'+c$.
\end{description}
By \cite[Theorems~2.3 and~2.6]{Tarski1949b}, the postulates VI' and VIII hold in any cardinal algebra $\langle A,{+},{\sum}\rangle$,
so every cardinal algebra is a weak cardinal algebra. It is shown in \cite[Section~6]{Truss1984} that
there is a weak cardinal algebra for which the cancellation law ``$2\cdot a=2\cdot b$ implies $a=b$'' fails.

For cardinals $\mathfrak{a}$ and $\mathfrak{b}$, we write $\mathfrak{a}=^\ast\mathfrak{b}$ if
there are sets $A$ and $B$ of cardinalities $\mathfrak{a}$ and $\mathfrak{b}$, respectively,
such that there are partial surjections from $A$ onto $B$ and from $B$ onto $A$.
A \emph{surjective cardinal} is an equivalence class of cardinals under $=^\ast$.
Since this may be a proper class, we may employ ``Scott's trick'' to ensure that
the equivalence class is actually a set, namely
\[
[\mathfrak{a}]=\{\mathfrak{b}\mid\mathfrak{a}=^\ast\mathfrak{b}
\wedge\forall\mathfrak{c}(\mathfrak{a}=^\ast\mathfrak{c}\rightarrow\rank(\mathfrak{b})\leqslant\rank(\mathfrak{c}))\}.
\]
Surjective cardinals may alternatively be defined as Scott equivalence classes of sets under the relation $\approx^\ast$:
$A\approx^\ast B$ if there are partial surjections from $A$ onto $B$ and from $B$ onto $A$.
It is shown in \cite[Theorem~2.7]{Truss1984} that surjective cardinals form a weak cardinal algebra,
and in \cite[Corollary~3.7]{Truss1984} that the cancellation law for surjective cardinals holds,
that is, $m\cdot\mathfrak{a}=^\ast m\cdot\mathfrak{b}$ implies $\mathfrak{a}=^\ast\mathfrak{b}$
for all cardinals $\mathfrak{a},\mathfrak{b}$ and all nonzero natural numbers $m$.

It is asked by Truss (see \cite[p.~179]{Truss1984} or \cite[p.~604]{Truss1990})
whether surjective cardinals form a cardinal algebra. Of course, if the axiom of choice is assumed,
then surjective cardinals are essentially the same as cardinals and hence form a cardinal algebra.
So, this question makes sense only in the absence of the axiom of choice.
In this article, we give a negative solution to this question by showing that
$\mathsf{ZF}+\mathsf{DC}_\kappa$, where $\kappa$ is a fixed aleph,
cannot prove that surjective cardinals form a cardinal algebra.

Nevertheless, we improve Truss's result by showing that surjective cardinals form a \emph{surjective cardinal algebra},
which is by definition an algebraic system $\langle A,{+},{\sum}\rangle$ satisfying the postulates I--VII, with VI replaced by VI'.
``Surjective cardinal algebras'' were introduced simultaneously and independently by K.~P.~S.~Bhaskara Rao and R.~M.~Shortt
on the one hand, and by F.~Wehrung on the other hand in \cite{Rao1992,Wehrung1992}.
They call such algebras ``weak cardinal algebras''. Since the term ``weak cardinal algebra'' was already used by Truss
for a different kind of algebra, we use the term ``surjective cardinal algebra'' here.
Note that, assuming $\mathsf{DC}_\omega$, for surjective cardinal algebras,
the cancellation law ``$m\cdot a=m\cdot b$ implies $a=b$ for $m\in\omega\setminus\{0\}$''
already holds (see \cite[p.~157]{Rao1992} or \cite[Proposition~2.9]{Wehrung1992}).
So, our result also yields a choice-free proof of the cancellation law for surjective cardinals
(by the device discussed in \cite[pp.~240--242]{Tarski1949b} or \cite[p.~166]{Truss1984}).

The article is organized as follows. In the next section, we show that, assuming $\mathsf{AC}_\omega$,
surjective cardinals form a surjective cardinal algebra. In the third section,
we show that surjective cardinals may not form a cardinal algebra, even if $\mathsf{DC}_\kappa$ is assumed.
In the last section, we conclude the article by some remarks.

\section{Surjective cardinals form a surjective cardinal algebra}
Truss has already shown that surjective cardinals form a weak cardinal algebra (see \cite[Theorem~2.7]{Truss1984}),
and that a weaker version of the remainder postulate holds (see \cite[Lemma~3.3]{Truss1984}).
So, we only need to prove the full remainder postulate.
However, for the convenience of the reader, we shall present here a complete proof that
surjective cardinals form a surjective cardinal algebra.

To produce choice-free proofs in cardinal arithmetic,
we frequently use expressions like ``one can explicitly define'' in our formulations.
For example, when we state the Cantor--Bernstein theorem as
``from injections $f:A\to B$ and $g:B\to A$, one can explicitly define a bijection $h:A\to B$'',
we mean that one can define a class function $H$ without free variables such that,
whenever $f$ is an injection from $A$ into $B$ and $g$ is an injection from $B$ into $A$,
$H(f,g)$ is defined and is a bijection between $A$ and $B$.

\begin{lemma}\label{sh02}
From a set $A$ and two families $\langle B_n\rangle_{n\in\omega}$ and $\langle f_n\rangle_{n\in\omega}$
such that $A\cap B_n=\varnothing$ and $f_n:A\to A\cup B_n$ is a partial surjection for all $n\in\omega$,
one can explicitly define a partial surjection $g:A\to A\cup\bigcup_{n\in\omega}B_n$.
\end{lemma}
\begin{proof}
Let $\pi$ be the Cantor pairing function, that is,
the bijection between $\omega\times\omega$ and $\omega$ defined by
\[
\pi(m,n)=\frac{(m+n)(m+n+1)}{2}+m.
\]
Define by recursion
\begin{align*}
h_0              & =\mathrm{id}_A,\\
h_{\pi(0,n)+1}   & =f_n\circ h_{\pi(0,n)},\\
h_{\pi(m+1,n)+1} & =h_{\pi(m,n)+1}\circ h_{\pi(m+1,n)}.
\end{align*}
An easy induction shows that, for all $m,n\in\omega$, $h_{\pi(m,n)+1}$ is a partial surjection from $A$ onto $A\cup B_n$.
For all $m,n\in\omega$, let
\[
C_{m,n}=h_{\pi(m,n)+1}^{-1}[B_n].
\]
An easy induction shows that, for all $k,l\in\omega$ with $k<l$, $h_l=h\circ h_k$ for some function $h$ with $\dom(h)\subseteq A$.
For all $m,n,m',n'$ with $\pi(m,n)<\pi(m',n')$, there is a function $h$ with $\dom(h)\subseteq A$
such that $h_{\pi(m',n')+1}=h\circ h_{\pi(m,n)+1}$, and hence
\[
C_{m',n'}=h_{\pi(m',n')+1}^{-1}[B_{n'}]=h_{\pi(m,n)+1}^{-1}[h^{-1}[B_{n'}]]\subseteq h_{\pi(m,n)+1}^{-1}[A],
\]
which implies $C_{m',n'}\cap C_{m,n}=\varnothing$ since $A\cap B_n=\varnothing$. Since for all $m,n\in\omega$ we have
\begin{align*}
h_{\pi(0,n)+1}[C_{0,n}]   & =B_n,\\
h_{\pi(m+1,n)}[C_{m+1,n}] & =h_{\pi(m+1,n)}[h_{\pi(m+1,n)+1}^{-1}[B_n]]=h_{\pi(m,n)+1}^{-1}[B_n]=C_{m,n},
\end{align*}
it is sufficient to define
\[
g=\bigcup_{n\in\omega}\bigl(h_{\pi(0,n)+1}\rstr C_{0,n}\bigr)\cup\bigcup_{m,n\in\omega}\bigl(h_{\pi(m+1,n)}\rstr C_{m+1,n}\bigr)
\cup\mathrm{id}_{A\setminus\bigcup_{m,n\in\omega}C_{m,n}}.\qedhere
\]
\end{proof}

\begin{corollary}\label{sh03}
From a set $A$ and a function $f$ such that $A\subseteq f[A]$,
one can explicitly define a partial surjection $g:A\to\bigcup_{n\in\omega}f^n[A]$.
\end{corollary}
\begin{proof}
Take $B_n=f^n[A]\setminus A$ and $f_n=f^n\rstr A$ in Lemma~\ref{sh02}.
\end{proof}

\begin{lemma}[Knaster's fixed point theorem]\label{sh01}
Let $i:\scrP(A)\to\scrP(A)$ be isotone. Then
\[
X=\textstyle\bigcup\{D\subseteq A\mid D\subseteq i(D)\}
\]
is a fixed point of $i$.
\end{lemma}
\begin{proof}
For every $D\subseteq A$ with $D\subseteq i(D)$, we have $D\subseteq X$,
and thus $D\subseteq i(D)\subseteq i(X)$ since $i$ is isotone. Hence, $X\subseteq i(X)$,
which implies $i(X)\subseteq i(i(X))$ since $i$ is isotone,
and so $i(X)\subseteq X$ by the definition of $X$. Therefore, $i(X)=X$.
\end{proof}

\begin{definition}
$\langle f,g\rangle$ is a \emph{surjection pair} between $A$ and $B$ if
$f:A\to B$ and $g:B\to A$ are partial surjections.
\end{definition}

The key step of our proof is the following lemma, which is Lemma~2.3 of \cite{Truss1984}.
The proof presented here is simpler and more straightforward than, but similar to, the one in \cite{Truss1984}.

\begin{lemma}\label{sh04}
From sets $A,B,C$ with $A\cap B=\varnothing$ and a surjection pair $\langle f,g\rangle$ between $A\cup B$ and $C$,
one can explicitly partition $A,B,C$ as
\begin{align*}
A & =A'\cup P\\
B & =B'\cup Q\\
C & =\tilde{A}\cup\tilde{B}
\end{align*}
and explicitly define partial surjections from $A'$ onto $A'\cup Q$ and from $B'$ onto $B'\cup P$,
and surjection pairs between $\tilde{A}$ and $A'$ and between $\tilde{B}$ and $B'$.
\end{lemma}
\begin{proof}
Without loss of generality, suppose that $f[A]\cap f[B]=\varnothing$.
Consider the isotone functions $i:\scrP(A)\to\scrP(A)$ and $j:\scrP(B)\to\scrP(B)$ defined by
\begin{align*}
i(D) & =A\cap g[f[D]],\\
j(E) & =B\cap g[f[E]].
\end{align*}
By Lemma~\ref{sh01},
\begin{align*}
X & =\textstyle\bigcup\{D\subseteq A\mid D\subseteq i(D)\}
\intertext{and}
Y & =\textstyle\bigcup\{E\subseteq B\mid E\subseteq j(E)\}
\end{align*}
are fixed points of $i$ and $j$, respectively. Let
\begin{align*}
f' & =f\setminus\bigl(f\rstr(f^{-1}[f[X]]\setminus X)\bigr),\\
g' & =g\setminus\bigl(g\rstr(g^{-1}[X]\setminus f[X])\bigr).
\end{align*}
Clearly, $\langle f',g'\rangle$ is a surjection pair between $A\cup B$ and $C$.
It is also easy to see that $X\subseteq g'[f'[X]]$, $Y\subseteq g'[f'[Y]]$, and ${f'}^{-1}[{g'}^{-1}[X]]\subseteq X$. Let
\begin{align*}
\tilde{f} & =f'\setminus\bigl(f'\rstr({f'}^{-1}[f'[Y]]\setminus Y)\bigr),\\
\tilde{g} & =g'\setminus\bigl(g'\rstr({g'}^{-1}[Y]\setminus f'[Y])\bigr).
\end{align*}
Clearly, $\langle\tilde{f},\tilde{g}\rangle$ is a surjection pair between $A\cup B$ and $C$.
It is also easy to see that \text{$X\subseteq \tilde{g}[\tilde{f}[X]]$}, $Y\subseteq \tilde{g}[\tilde{f}[Y]]$,
$\tilde{f}^{-1}[\tilde{g}^{-1}[X]]\subseteq X$, and $\tilde{f}^{-1}[\tilde{g}^{-1}[Y]]\subseteq Y$.

We claim that, for every $c\in C$,
\begin{align}
(A\cup Y)\cap\bigcup_{n\in\omega}(\tilde{g}\circ\tilde{f})^{-n}[\tilde{f}^{-1}[\{c\}]] & \neq\varnothing,\label{sh05}\\
(B\cup X)\cap\bigcup_{n\in\omega}(\tilde{g}\circ\tilde{f})^{-n}[\tilde{f}^{-1}[\{c\}]] & \neq\varnothing.\label{sh06}
\end{align}
Assume to the contrary that $(A\cup Y)\cap\bigcup_{n\in\omega}(\tilde{g}\circ\tilde{f})^{-n}[\tilde{f}^{-1}[\{c\}]]=\varnothing$
for some $c\in C$. Let $E=Y\cup\bigcup_{n\in\omega}(\tilde{g}\circ\tilde{f})^{-n}[\tilde{f}^{-1}[\{c\}]]\subseteq B$.
It is easy to see that $E\subseteq j(E)$, so $E\subseteq Y$, a contradiction.
This proves \eqref{sh05}. The proof of \eqref{sh06} is similar.

Now, we define
\begin{align*}
P  & =A\cap\bigcup_{n\in\omega}(\tilde{g}\circ\tilde{f})^n[Y],\\
Q  & =B\cap\bigcup_{n\in\omega}(\tilde{g}\circ\tilde{f})^n[X],\\
A' & =A\setminus P,\\
B' & =B\setminus Q.
\end{align*}
Using $\tilde{f}^{-1}[\tilde{g}^{-1}[X]]\subseteq X$, an easy induction shows that
$X\cap(\tilde{g}\circ\tilde{f})^n[Y]=\varnothing$ for all $n\in\omega$, so $X\cap P=\varnothing$,
which implies $X\subseteq A'$. Since $X\subseteq \tilde{g}[\tilde{f}[X]]$, it follows from Corollary~\ref{sh03} that
one can explicitly define a partial surjection from $X$ onto $\bigcup_{n\in\omega}(\tilde{g}\circ\tilde{f})^n[X]$,
which includes $X\cup Q$. Hence, one can explicitly define a partial surjection from $A'$ onto $A'\cup Q$,
and similarly a partial surjection from $B'$ onto $B'\cup P$.

Finally, we define
\begin{align*}
\tilde{A} & =\bigl(\tilde{g}^{-1}[A]\setminus\bigcup_{n\in\omega}\tilde{f}[(\tilde{g}\circ\tilde{f})^n[Y]]\bigr)
\cup\bigcup_{n\in\omega}\tilde{f}[(\tilde{g}\circ\tilde{f})^n[X]]\cup\bigl(\tilde{f}[A']\setminus\dom(\tilde{g})\bigr),\\
\tilde{B} & =C\setminus\tilde{A}.
\end{align*}
The situation is illustrated in Figure~\ref{sh15}.

\begin{figure}[htb]
\begin{tikzpicture}
\draw (0,0) ellipse (25pt and 50pt);
\draw (0,-4) ellipse (25pt and 50pt);
\draw (0,0.5) ellipse (15pt and 25pt);
\draw (0,-4.5) ellipse (15pt and 25pt);
\draw (4,-2) ellipse (40pt and 100pt);
\draw (2.6,-2) .. controls (3.6,-2.6) and (4.4,-2.6) .. (5.4,-2);
\draw[-] (-0.73,-1)--node[below]{$P$}node[above]{$A'$}(0.73,-1);
\draw[-] (-0.73,-3)--node[above]{$Q$}node[below]{$B'$}(0.73,-3);
\draw[->] (1,1.5) to[bend left] node[above]{$\tilde{f}$} (3,1.4);
\draw[->] (3,-5.4) to[bend left] node[below]{$\tilde{g}$} (1,-5.5);
\node at (0,0.7) {$X$};
\node at (0,-4.7) {$Y$};
\node at (-0.9,1.5) {$A$};
\node at (-0.9,-5.5) {$B$};
\node at (4,-0.5) {$\tilde{A}$};
\node at (4,-3.5) {$\tilde{B}$};
\node at (5,1.5) {$C$};
\end{tikzpicture}
\caption{The situation between these sets.}\label{sh15}
\end{figure}
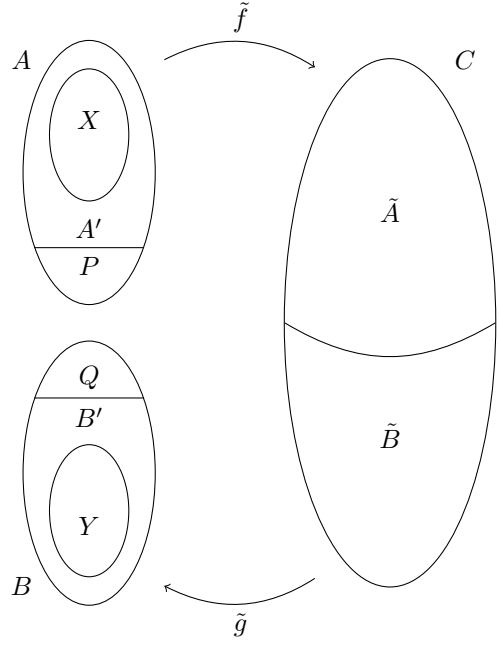

We first note that
\begin{equation}\label{sh07}
\tilde{g}^{-1}[B]\setminus\bigcup_{n\in\omega}\tilde{f}[(\tilde{g}\circ\tilde{f})^n[X]]\subseteq\tilde{B}
\end{equation}
and
\begin{equation}\label{sh08}
\tilde{B}\subseteq\bigl(\tilde{g}^{-1}[B]\setminus\bigcup_{n\in\omega}\tilde{f}[(\tilde{g}\circ\tilde{f})^n[X]]\bigr)
\cup\bigcup_{n\in\omega}\tilde{f}[(\tilde{g}\circ\tilde{f})^n[Y]]\cup\bigl(\tilde{f}[B']\setminus\dom(\tilde{g})\bigr).
\end{equation}
It is also easy to see that
\begin{align*}
\tilde{g}^{-1}[A'] & \subseteq\tilde{g}^{-1}[A]\setminus\bigcup_{n\in\omega}\tilde{f}[(\tilde{g}\circ\tilde{f})^n[Y]],\\
\tilde{g}^{-1}[B'] & \subseteq\tilde{g}^{-1}[B]\setminus\bigcup_{n\in\omega}\tilde{f}[(\tilde{g}\circ\tilde{f})^n[X]].
\end{align*}
So, $\tilde{g}$ induces partial surjections from $\tilde{A}$ onto $A'$ and from $\tilde{B}$ onto $B'$ by \eqref{sh07}.

We conclude the proof by explicitly defining partial surjections from $A'$ onto $\tilde{A}$ and from $B'$ onto $\tilde{B}$ as follows.
Since $X\subseteq A'$, it follows that
\begin{equation}\label{sh09}
\tilde{g}^{-1}[X]\subseteq\tilde{g}^{-1}[A']\subseteq\tilde{g}^{-1}[A]\setminus\bigcup_{n\in\omega}\tilde{f}[(\tilde{g}\circ\tilde{f})^n[Y]].
\end{equation}
Since $\tilde{f}^{-1}[\tilde{g}^{-1}[X]]\subseteq X$, we have
$\tilde{g}^{-1}[X]\subseteq\tilde{f}[X]=\tilde{f}[\tilde{g}[\tilde{g}^{-1}[X]]]$, so
it follows from Corollary~\ref{sh03} that one can explicitly define a partial surjection
from $\tilde{g}^{-1}[X]$ onto $\bigcup_{n\in\omega}(\tilde{f}\circ\tilde{g})^n[\tilde{g}^{-1}[X]]
=\tilde{g}^{-1}[X]\cup\bigcup_{n\in\omega}\tilde{f}[(\tilde{g}\circ\tilde{f})^n[X]]$, which implies that, by \eqref{sh09},
\begin{equation}\label{sh10}
\parbox{33em}{
one can explicitly define a partial surjection from
$\tilde{g}^{-1}[A]\setminus\bigcup_{n\in\omega}\tilde{f}[(\tilde{g}\circ\tilde{f})^n[Y]]$ onto
$\bigl(\tilde{g}^{-1}[A]\setminus\bigcup_{n\in\omega}\tilde{f}[(\tilde{g}\circ\tilde{f})^n[Y]]\bigr)
\cup\bigcup_{n\in\omega}\tilde{f}[(\tilde{g}\circ\tilde{f})^n[X]]$.}
\end{equation}

Let $h$ be the partial function on $A'$ defined by
\[
h(a)=
\begin{cases}
\tilde{f}((\tilde{g}\circ\tilde{f})^m(a)) & \text{if there exists a least natural number $m$ for which}\\
& \tilde{f}((\tilde{g}\circ\tilde{f})^m(a))\in\tilde{g}^{-1}[A]\setminus\bigcup_{n\in\omega}\tilde{f}[(\tilde{g}\circ\tilde{f})^n[Y]],\\
\tilde{f}(a)     & \text{if $a\in\dom(\tilde{f})$ and $\tilde{f}(a)\notin\dom(\tilde{g})$,}\\
\text{undefined} & \text{otherwise.}
\end{cases}
\]
We claim that
\begin{equation}\label{sh11}
\bigl(\tilde{g}^{-1}[A]\setminus\bigcup_{n\in\omega}\tilde{f}[(\tilde{g}\circ\tilde{f})^n[Y]]\bigr)
\cup\bigl(\tilde{f}[A']\setminus\dom(\tilde{g})\bigr)\subseteq\ran(h).
\end{equation}
Clearly, $\tilde{f}[A']\setminus\dom(\tilde{g})\subseteq\ran(h)$.
Let $c\in\tilde{g}^{-1}[A]\setminus\bigcup_{n\in\omega}\tilde{f}[(\tilde{g}\circ\tilde{f})^n[Y]]$.
By \eqref{sh05}, it is easy to see that $A'\cap\bigcup_{n\in\omega}(\tilde{g}\circ\tilde{f})^{-n}[\tilde{f}^{-1}[\{c\}]]\neq\varnothing$.
Let $m$ be the least natural number for which $A'\cap(\tilde{g}\circ\tilde{f})^{-m}[\tilde{f}^{-1}[\{c\}]]\neq\varnothing$.
Let $a\in A'\cap(\tilde{g}\circ\tilde{f})^{-m}[\tilde{f}^{-1}[\{c\}]]$. Then
\[
\tilde{f}((\tilde{g}\circ\tilde{f})^m(a))=c\in\tilde{g}^{-1}[A]\setminus\bigcup_{n\in\omega}\tilde{f}[(\tilde{g}\circ\tilde{f})^n[Y]].
\]
If there exists an $l<m$ such that
$\tilde{f}((\tilde{g}\circ\tilde{f})^l(a))\in\tilde{g}^{-1}[A]\setminus\bigcup_{n\in\omega}\tilde{f}[(\tilde{g}\circ\tilde{f})^n[Y]]$,
then it is easy to see that $(\tilde{g}\circ\tilde{f})^{l+1}(a)\in A'\cap(\tilde{g}\circ\tilde{f})^{-(m-l-1)}[\tilde{f}^{-1}[\{c\}]]$,
contradicting the minimality of $m$. Hence, $c=\tilde{f}((\tilde{g}\circ\tilde{f})^m(a))=h(a)\in\ran(h)$.

Now, by \eqref{sh10} and \eqref{sh11}, one can explicitly define a partial surjection from $A'$ onto $\tilde{A}$.
Similarly, by \eqref{sh08}, one can explicitly define a partial surjection from $B'$ onto $\tilde{B}$.
\end{proof}

\begin{lemma}\label{sh16}
From pairwise disjoint sets $D_1,D_2,Q$ and a partial surjection $f$ from $D_1\cup D_2$ onto $D_1\cup D_2\cup Q$,
one can explicitly partition $Q$ as $Q=Q_1\cup Q_2$ and explicitly define partial surjections
from $D_1$ onto $D_1\cup Q_1$ and from $D_2$ onto $D_2\cup Q_2$.
\end{lemma}
\begin{proof}
Let
\begin{align*}
C_1 & =\{c\in D_1\mid f^{-n}[\{c\}]\subseteq D_1\text{ for all }n\in\omega\},\\
C_2 & =\{c\in D_2\mid f^{-n}[\{c\}]\cap C_1=\varnothing\text{ for all }n\in\omega\}.
\end{align*}
Let $g_1$ and $g_2$ be the functions on $D_1$ and $D_2$, respectively, defined by
\begin{align*}
g_1(c) & =
\begin{cases}
f(c)   & \text{if $c\in C_1$ and $f(c)\in C_1$,}\\
f^m(c) & \text{if $c\in C_1\setminus f^{-1}[C_1]$ and $f^m(c)\in Q$ for some least $m>0$,}\\
c      & \text{otherwise,}
\end{cases}
\intertext{and}
g_2(c) & =
\begin{cases}
f^k(c) & \text{if $c\in C_2$ and $f^k(c)\in C_2$ for some least $k>0$,}\\
f^l(c) & \text{if $c\in C_2\setminus\bigcup_{n>0}f^{-n}[C_2]$ and $f^l(c)\in Q$ for some least $l>0$,}\\
c      & \text{otherwise.}
\end{cases}
\end{align*}
It is easy to see that $D_i\subseteq\ran(g_i)$ for $i=1,2$.
Let $Q_1=Q\cap\ran(g_1)$ and let $Q_2=Q\setminus Q_1$.
Then $g_1$ is a surjection from $D_1$ onto $D_1\cup Q_1$.
It suffices to show $Q_2\subseteq\ran(g_2)$, since then $g_2$ will induce a partial surjection from $D_2$ onto $D_2\cup Q_2$.
Let $e\in Q_2$. Since $e\notin\ran(g_1)$, it follows that $e=f^m(c)$ for no $c\in C_1$ and $m\in\omega$,
and hence there is a least $l>0$ such that $e=f^l(d)$ for some $d\in C_2$.
By the minimality of $l$, we have $d\notin\bigcup_{n>0}f^{-n}[C_2]$, so $e=f^l(d)=g_2(d)\in\ran(g_2)$.
\end{proof}

\begin{lemma}\label{sh17}
From pairwise disjoint sets $A_1,A_2,B_1,B_2$ and a surjection pair $\langle f,g\rangle$ between $A_1\cup A_2$ and $B_1\cup B_2$,
one can explicitly define pairwise disjoint sets $C_1,C_2,C_3,C_4$ and surjection pairs
between $A_1$ and $C_1\cup C_2$, between $A_2$ and $C_3\cup C_4$, between $B_1$ and $C_1\cup C_3$, and between $B_2$ and $C_2\cup C_4$.
\end{lemma}
\begin{proof}
By Lemma~\ref{sh04}, one can explicitly partition $A_1,A_2,B_1\cup B_2$ as
\begin{align*}
A_1         & =A_1'\cup P\\
A_2         & =A_2'\cup Q\\
B_1\cup B_2 & =\tilde{A}_1\cup\tilde{A}_2
\end{align*}
and explicitly define partial surjections $h_1:A_1'\to A_1'\cup Q$ and $h_2:A_2'\to A_2'\cup P$,
a~surjection pair $\langle f_1,g_1\rangle$ between $\tilde{A}_1$ and $A_1'$,
and a surjection pair $\langle f_2,g_2\rangle$ between $\tilde{A}_2$ and $A_2'$. Let
\begin{align*}
D_1 & =\tilde{A}_1\cap B_1,\\
D_2 & =\tilde{A}_1\cap B_2,\\
D_3 & =\tilde{A}_2\cap B_1,\\
D_4 & =\tilde{A}_2\cap B_2.
\end{align*}
Since $(g_1\cup\mathrm{id}_Q)\circ h_1\circ f_1$ and $(g_2\cup\mathrm{id}_P)\circ h_2\circ f_2$
are partial surjections from $D_1\cup D_2$ onto $D_1\cup D_2\cup Q$
and from $D_3\cup D_4$ onto $D_3\cup D_4\cup P$, respectively,
it follows from Lemma~\ref{sh16} that one can explicitly partition $P,Q$ as
\begin{align*}
P & =P_1\cup P_2\\
Q & =Q_1\cup Q_2
\end{align*}
and explicitly define partial surjections $s_1:D_1\to D_1\cup Q_1$, $s_2:D_2\to D_2\cup Q_2$,
$s_3:D_3\to D_3\cup P_1$, and $s_4:D_4\to D_4\cup P_2$. Finally, we define
\begin{align*}
C_1 & =D_1\cup P_1,\\
C_2 & =D_2\cup P_2,\\
C_3 & =D_3\cup Q_1,\\
C_4 & =D_4\cup Q_2.
\end{align*}
Then $C_1,C_2,C_3,C_4$ are pairwise disjoint,
and $\langle g_1\cup\mathrm{id}_P,f_1\cup\mathrm{id}_P\rangle$, $\langle g_2\cup\mathrm{id}_Q,f_2\cup\mathrm{id}_Q\rangle$,
$\langle s_1\cup s_3,\mathrm{id}_{D_1}\cup\mathrm{id}_{D_3}\rangle$, and $\langle s_2\cup s_4,\mathrm{id}_{D_2}\cup\mathrm{id}_{D_4}\rangle$
are surjection pairs between $A_1$ and $C_1\cup C_2$, between $A_2$ and $C_3\cup C_4$,
between $B_1$ and $C_1\cup C_3$, and between $B_2$ and $C_2\cup C_4$, respectively.
\end{proof}

The next corollary immediately follows from Lemma~\ref{sh17}.

\begin{corollary}\label{sh18}
The finite refinement postulate holds for surjective cardinals, that is,
for all cardinals $\mathfrak{a}_1,\mathfrak{a}_2,\mathfrak{b}_1,\mathfrak{b}_2$,
if $\mathfrak{a}_1+\mathfrak{a}_2=^\ast\mathfrak{b}_1+\mathfrak{b}_2$,
then there are cardinals $\mathfrak{c}_1,\mathfrak{c}_2,\mathfrak{c}_3,\mathfrak{c}_4$ such that
$\mathfrak{a}_1=^\ast\mathfrak{c}_1+\mathfrak{c}_2$, $\mathfrak{a}_2=^\ast\mathfrak{c}_3+\mathfrak{c}_4$,
$\mathfrak{b}_1=^\ast\mathfrak{c}_1+\mathfrak{c}_3$, and $\mathfrak{b}_2=^\ast\mathfrak{c}_2+\mathfrak{c}_4$.
\end{corollary}

\begin{lemma}\label{sh12}
From pairwise disjoint sets $A_n,B_n$ {\upshape($n\in\omega$)} and surjection pairs $\langle f_n,g_n\rangle$ \text{{\upshape($n\in\omega$)}}
between $A_n$ and $A_{n+1}\cup B_n$, one can explicitly define a set $C$ disjoint from $\bigcup_{n\in\omega}B_n$
and a surjection pair between $A_m$ and $C\cup\bigcup_{n\in\omega}B_{m+n}$ for~each $m\in\omega$.
\end{lemma}
\begin{proof}
We define sets $\tilde{A}_n,\tilde{B}_n,A_n',B_n',P_n,Q_n$ and functions $\tilde{f}_n,\tilde{g}_n,f_n',g_n',p_n,q_n$ as follows.
Let $\tilde{A}_0=A_0'=A_0$, $P_0=\varnothing$, and $\tilde{f}_0=\tilde{g}_0=\mathrm{id}_{A_0}$.
Assume $\tilde{A}_n,A_n',P_n,\tilde{f}_n,\tilde{g}_n$ have been defined so that
$\tilde{A}_n\cap(P_n\cup\bigcup_{k>n}A_k)=\varnothing$, $A_n'\cap P_n=\varnothing$, $A_n=A_n'\cup P_n$,
and $\langle\tilde{f}_n,\tilde{g}_n\rangle$ is a surjection pair between $\tilde{A}_n$ and $A_n'$.
Since $\langle(\tilde{g}_n\cup\mathrm{id}_{P_n})\circ g_n,f_n\circ(\tilde{f}_n\cup\mathrm{id}_{P_n})\rangle$ is a surjection pair
between $A_{n+1}\cup B_n$ and $\tilde{A}_n\cup P_n$, it follows from Lemma~\ref{sh04} that
one can explicitly partition $A_{n+1},B_n,\tilde{A}_n\cup P_n$ as
\begin{align*}
A_{n+1}             & =A_{n+1}'\cup P_{n+1}\\
B_n                 & =B_n'\cup Q_n\\
\tilde{A}_n\cup P_n & =\tilde{A}_{n+1}\cup\tilde{B}_n
\end{align*}
and define partial surjections $q_n:A_{n+1}'\to A_{n+1}'\cup Q_n$ and $p_n:B_n'\to B_n'\cup P_{n+1}$,
a~surjection pair $\langle\tilde{f}_{n+1},\tilde{g}_{n+1}\rangle$ between $\tilde{A}_{n+1}$ and $A_{n+1}'$,
and a surjection pair $\langle f_n',g_n'\rangle$ between $\tilde{B}_n$ and $B_n'$.
Clearly, $\tilde{A}_{n+1}\cap(P_{n+1}\cup\bigcup_{k>n+1}A_k)=\varnothing$.
An easy induction shows that, for all $m,n\in\omega$,
\begin{equation}\label{sh13}
\tilde{A}_{m+n+1},\tilde{B}_{m+n}\subseteq\tilde{A}_m\cup\bigcup_{k\leqslant n}P_{m+k}.
\end{equation}
Since $\tilde{A}_{n+1}\cap\tilde{B}_n=\varnothing$ for all $n\in\omega$,
it follows from \eqref{sh13} that $\tilde{B}_n$ ($n\in\omega$) are pairwise disjoint.
Also, by \eqref{sh13}, $\bigcup_{n\in\omega}\tilde{B}_n\subseteq\bigcup_{n\in\omega}A_n$,
and hence $\bigcup_{n\in\omega}\tilde{B}_n\cap\bigcup_{n\in\omega}B_n=\varnothing$.

For each $m\in\omega$, let
\[
D_m=\tilde{A}_m\cup\bigcup_{n\in\omega}P_{m+n}.
\]
Clearly, for every $m\in\omega$, $D_{m+1}\cap\tilde{B}_m=\varnothing$ and $D_m=D_{m+1}\cup\tilde{B}_m$. Now, we define
\[
C=\bigcap_{m\in\omega}D_m.
\]
Since $C\subseteq D_0\subseteq\bigcup_{n\in\omega}A_{n}$, $C\cap\bigcup_{n\in\omega}B_n=\varnothing$.
Note also that $C\cap\bigcup_{n\in\omega}\tilde{B}_n=\varnothing$, and for every $m\in\omega$,
\begin{equation}\label{sh14}
D_m=C\cup\bigcup_{n\in\omega}\tilde{B}_{m+n}.
\end{equation}

Let $m\in\omega$. We conclude the proof by explicitly defining a surjection pair
between $A_m$ and $C\cup\bigcup_{n\in\omega}B_{m+n}$ as follows. By \eqref{sh14},
\[
\mathrm{id}_C\cup\bigcup_{n\in\omega}g_{m+n}':C\cup\bigcup_{n\in\omega}B_{m+n}\to D_m
\]
is a partial surjection, so is
\[
(\tilde{f}_m\cup\mathrm{id}_{P_m})\circ(\mathrm{id}_C\cup\bigcup_{n\in\omega}g_{m+n}'):C\cup\bigcup_{n\in\omega}B_{m+n}\to A_m.
\]
Also, by \eqref{sh14},
\[
\mathrm{id}_C\cup\bigcup_{n\in\omega}f_{m+n}'\cup\bigcup_{n\in\omega}\mathrm{id}_{Q_{m+n}}:
D_m\cup\bigcup_{n\in\omega}Q_{m+n}\to C\cup\bigcup_{n\in\omega}B_{m+n}
\]
is a partial surjection, so is
\begin{multline*}
(\mathrm{id}_C\cup\bigcup_{n\in\omega}f_{m+n}'\cup\bigcup_{n\in\omega}\mathrm{id}_{Q_{m+n}})
\circ(\tilde{g}_m\cup\bigcup_{n\in\omega}\mathrm{id}_{P_{m+n}}\cup\bigcup_{n\in\omega}\mathrm{id}_{Q_{m+n}}):\\
A_m\cup\bigcup_{n\in\omega}P_{m+n+1}\cup\bigcup_{n\in\omega}Q_{m+n}\to C\cup\bigcup_{n\in\omega}B_{m+n}.
\end{multline*}
Therefore, it is sufficient to explicitly define a partial surjection from $A_m$
onto $A_m\cup\bigcup_{n\in\omega}P_{m+n+1}\cup\bigcup_{n\in\omega}Q_{m+n}$.
By Lemma~\ref{sh02}, it suffices to explicitly define partial surjections
from $A_m$ onto $A_m\cup P_{m+n+1}$ and from $A_m$ onto $A_m\cup Q_{m+n}$ for~each $n\in\omega$.

Let $n\in\omega$. For each $l\leqslant n$, let
\begin{align*}
f_{m+l}'' & =f_{m+l}\cup\bigcup_{k<l}\mathrm{id}_{B_{m+k}},\\
g_{m+l}'' & =g_{m+l}\cup\bigcup_{k<l}\mathrm{id}_{B_{m+k}}.
\end{align*}
Then, for each $l\leqslant n$, $\langle f_{m+l}'',g_{m+l}''\rangle$ is a surjection pair between
$A_{m+l}\cup\bigcup_{k<l}B_{m+k}$ and $A_{m+l+1}\cup\bigcup_{k\leqslant l}B_{m+k}$,
so $\langle f_{m+n}''\circ f_{m+n-1}''\circ\dots\circ f_m'',g_m''\circ g_{m+1}''\circ\dots\circ g_{m+n}''\rangle$
is a surjection pair between $A_m$ and $A_{m+n+1}\cup\bigcup_{k\leqslant n}B_{m+k}$,
from which (as well as $p_{m+n},q_{m+n}$) one can explicitly define partial surjections
from $A_m$ onto $A_m\cup P_{m+n+1}$ and from $A_m$ onto $A_m\cup Q_{m+n}$.
\end{proof}

The next corollary immediately follows from Lemma~\ref{sh12},
which is a generalization of \cite[Lemma~3.3]{Truss1984}.

\begin{corollary}\label{sh19}
{\upshape($\mathsf{AC}_\omega$)} The remainder postulate holds for surjective cardinals, that is,
for all cardinals $\mathfrak{a}_n,\mathfrak{b}_n$ {\upshape($n\in\omega$)},
if $\mathfrak{a}_n=^\ast\mathfrak{a}_{n+1}+\mathfrak{b}_n$ for all $n\in\omega$,
then there is a cardinal $\mathfrak{c}$ such that
\[
\mathfrak{a}_m=^\ast\mathfrak{c}+\sum_{n\in\omega}\mathfrak{b}_{m+n}\text{ for all }m\in\omega.
\]
\end{corollary}

\begin{theorem}
{\upshape($\mathsf{AC}_\omega$)} Surjective cardinals form a surjective cardinal algebra.
\end{theorem}
\begin{proof}
The postulates I--V hold obviously, and the postulates VI' and VII hold for surjective cardinals
by Corollaries~\ref{sh18} and~\ref{sh19}, respectively.
\end{proof}

\begin{corollary}\label{sh26}
The cancellation law holds for surjective cardinals, that is,
for all cardinals $\mathfrak{a},\mathfrak{b}$ and all nonzero natural numbers $m$,
if $m\cdot\mathfrak{a}=^\ast m\cdot\mathfrak{b}$, then $\mathfrak{a}=^\ast\mathfrak{b}$.
\end{corollary}
\begin{proof}
It is noted by Bhaskara Rao and Shortt \cite[p.~157]{Rao1992} and proved by Wehrung \cite[Proposition~2.9]{Wehrung1992}
(even for a more general kind of algebra) that, assuming $\mathsf{DC}_\omega$, the cancellation law holds for
any surjective cardinal algebra. Go through the algebraic proof of the cancellation law for surjective cardinals,
transfer each intermediate step to the corresponding explicit-definability version using Lemmas~\ref{sh17} and~\ref{sh12},
and finally a choice-free proof of the cancellation law will be obtained.
\end{proof}

In the next section, we show that the refinement postulate may fail for surjective cardinals,
even if $\mathsf{DC}_\kappa$ is assumed, where $\kappa$ is a fixed aleph.

\section{Surjective cardinals may not form a cardinal algebra}
Let $\kappa$ be an aleph. We shall prove that it is consistent with $\mathsf{ZF}+\mathsf{DC}_\kappa$
that surjective cardinals do not form a cardinal algebra.
We shall employ the method of permutation models.

We refer the reader to~\cite[Chap.~8]{Halbeisen2017} or~\cite[Chap.~4]{Jech1973}
for an introduction to the theory of permutation models.
Permutation models are not models of $\mathsf{ZF}$;
they are models of $\mathsf{ZFA}$ (the Zermelo--Fraenkel set theory with atoms).
We shall construct a permutation model in which $\mathsf{DC}_\kappa$ holds but
surjective cardinals do not form a cardinal algebra. Then, by a transfer theorem of Pincus \cite[Theorem~4]{Pincus1977},
we conclude that $\mathsf{ZF}+\mathsf{DC}_\kappa$ cannot prove that surjective cardinals form a cardinal algebra.

We work in $\mathsf{ZFC}$, and construct the set $A$ of atoms as follows.
\[
A=\bigcup_{n\in\omega}A_n,
\]
where
\[
A_n=\{\langle\alpha,i,n\rangle\mid\alpha<\kappa^+\text{ and }i<2\}.
\]
Let $\mathcal{G}$ be the group of all permutations $\tau$ of $A$ such that
for each $\alpha<\kappa^+$ there is a permutation $p_\alpha$ of $\{0,1\}$ such that
$\tau(\langle\alpha,i,n\rangle)=\langle\alpha,p_\alpha(i),n\rangle$ for all $i<2$ and all $n\in\omega$.
In other words, $\mathcal{G}$ is the group of all permutations of $A$ that preserve
the tree structure of $A\cup\kappa^+$ illustrated in Figure~\ref{sh20}.
Then $x$ belongs to the permutation model $\mathcal{V}$ determined by $\mathcal{G}$ if and only if
$x\subseteq\mathcal{V}$ and $x$ has a \emph{support} of cardinality~$\leqslant\kappa$, that is, a subset $E\subseteq A$
with $|E|\leqslant\kappa$ such that every permutation $\tau\in\mathcal{G}$ fixing $E$ pointwise also fixes $x$.
Note that, for every $n\in\omega$, $A_n$ is fixed by every permutation in $\mathcal{G}$, so $A_n\in\mathcal{V}$.

\begin{figure}[htb]
\begin{tikzpicture}
\foreach \a in {0.5,2.5,4.5,7.5}
  \fill (\a,0) circle (1pt);
\foreach \b in {0,1,2,3,4,5,7,8}
\foreach \c in {-2,-4,-6}
  \fill (\b,\c) circle (1pt);
\foreach \d in {0,-2,-4,-6}
  \draw[dotted] (5.5,\d)--(6.5,\d);
\foreach \d in {0,-2,-4,-6}
  \draw[dotted] (8.5,\d)--(9.5,\d);
\foreach \b in {0,1,2,3,4,5,7,8}
  \draw[dotted] (\b,-8)--(\b,-7);
  \draw[->] (0,-2)--(0.5,0);
  \draw[->] (1,-2)--(0.5,0);
  \draw[->] (2,-2)--(2.5,0);
  \draw[->] (3,-2)--(2.5,0);
  \draw[->] (4,-2)--(4.5,0);
  \draw[->] (5,-2)--(4.5,0);
  \draw[->] (7,-2)--(7.5,0);
  \draw[->] (8,-2)--(7.5,0);
\foreach \b in {0,1,2,3,4,5,7,8}
  \draw[->] (\b,-4)--(\b,-2);
\foreach \b in {0,1,2,3,4,5,7,8}
  \draw[->] (\b,-6)--(\b,-4);
\foreach \b in {0,1,2,3,4,5,7,8}
  \draw[->] (\b,-7)--(\b,-6);
\foreach \d in {0,-2,-4,-6}
  \draw (4.75,\d) ellipse (150pt and 15pt);
  \node at (9.6,0.6) {$\kappa^+$};
  \node at (9.6,-1.4) {$A_0$};
  \node at (9.6,-3.4) {$A_1$};
  \node at (9.6,-5.4) {$A_2$};
  \node at (0.7,0.1) {$0$};
  \node at (2.7,0.1) {$1$};
  \node at (4.7,0.1) {$2$};
  \node at (7.7,0.1) {$\alpha$};
\end{tikzpicture}
\caption{The tree structure of $A\cup\kappa^+$.}\label{sh20}
\end{figure}
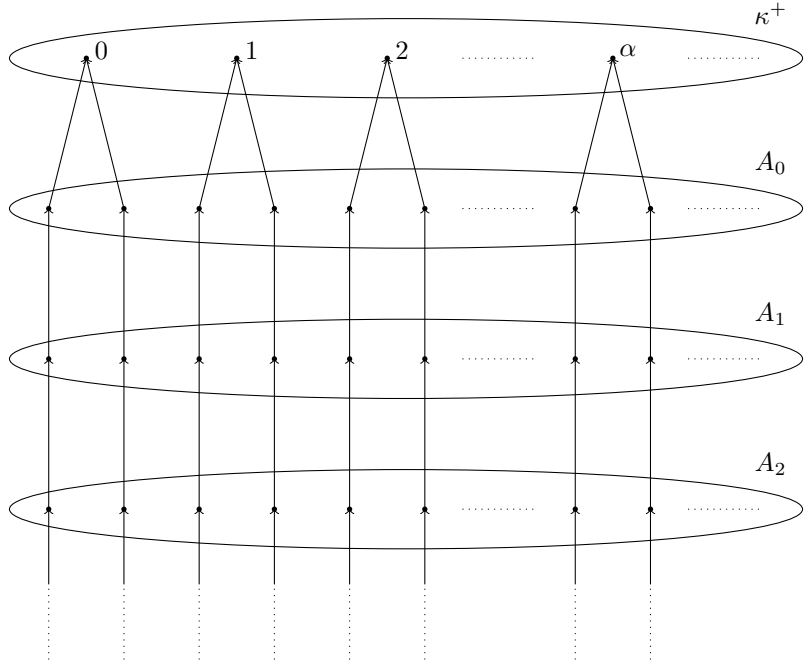

\begin{lemma}\label{sh21}
For every $\beta\leqslant\kappa$ and every function $g:\beta\to\mathcal{V}$, we have $g\in\mathcal{V}$.
\end{lemma}
\begin{proof}
For all $\alpha<\beta$, $g(\alpha)\in\mathcal{V}$ has a support $E_\alpha$ with $|E_\alpha|\leqslant\kappa$.
Let $E=\bigcup_{\alpha<\beta}E_\alpha$. Then $|E|\leqslant\kappa$ and $E$ supports each $g(\alpha)$, $\alpha<\beta$.
Thus, $E$ supports $g$, so $g\in\mathcal{V}$.
\end{proof}

\begin{lemma}\label{sh22}
In $\mathcal{V}$, $\mathsf{DC}_\kappa$ holds.
\end{lemma}
\begin{proof}
Let $S\in\mathcal{V}$ and let $R\in\mathcal{V}$ be a binary relation such that
for all $\alpha<\kappa$ and all $\alpha$-sequences $s\in\mathcal{V}$
of elements of $S$ there is $y\in S$ such that $sRy$. By Lemma~\ref{sh21},
for each $\alpha<\kappa$, every $\alpha$-sequence $s$ of elements of $S$ belongs to $\mathcal{V}$,
so $sRy$ for some $y\in S$. By the axiom of choice, there is a function $f:\kappa\to S$
such that $(f\rstr\alpha)Rf(\alpha)$ for every $\alpha<\kappa$.
By Lemma~\ref{sh21} again, $f\in\mathcal{V}$ and so $\mathsf{DC}_\kappa$ holds in $\mathcal{V}$.
\end{proof}

\begin{lemma}\label{sh23}
In $\mathcal{V}$, for every $m\in\omega$, there is no surjection from $A_m$ onto $A_m\cup\kappa^+$.
\end{lemma}
\begin{proof}
Let $m\in\omega$. Assume towards a contradiction that there is a surjection $f\in\mathcal{V}$ from $A_m$ onto $A_m\cup\kappa^+$.
Then $f$ has a support $E$ with $|E|\leqslant\kappa$. Let
\[
\tilde{E}=E\cup\{\langle\alpha,j,m\rangle\mid j<2\text{ and }\langle\alpha,i,n\rangle\in E\text{ for some }i<2\text{ and }n\in\omega\}.
\]
Clearly, $|\tilde{E}|\leqslant\kappa$, $\tilde{E}$ is a support of $f$, and for all $j<2$ and all $\langle\alpha,i,n\rangle\in\tilde{E}$,
$\langle\alpha,j,m\rangle\in\tilde{E}$.

We claim that, for all $\alpha<\kappa^+$ such that $\{\langle\alpha,0,m\rangle,\langle\alpha,1,m\rangle\}\nsubseteq\tilde{E}$,
\begin{equation}\label{sh24}
f[\{\langle\alpha,0,m\rangle,\langle\alpha,1,m\rangle\}]=\{\langle\alpha,0,m\rangle,\langle\alpha,1,m\rangle\}.
\end{equation}
Let $j<2$ be such that $\langle\alpha,j,m\rangle\notin\tilde{E}$.
Then $\langle\alpha,i,n\rangle\notin\tilde{E}$ for all $i<2$ and all $n\in\omega$.
Let $i<2$. Since $f$ is surjective, it follows that $\langle\alpha,i,m\rangle=f(\langle\alpha',i',m\rangle)$
for some $\alpha'<\kappa^+$ and $i'<2$. If $\alpha'\neq\alpha$, then there exists a permutation
$\tau\in\mathcal{G}$ that fixes $\tilde{E}\cup\{\langle\alpha',i',m\rangle\}$ pointwise and swaps
$\langle\alpha,i,m\rangle$ with $\langle\alpha,1-i,m\rangle$, contradicting that $\tilde{E}$ is a support of $f$.
So $\alpha'=\alpha$. Hence,
\[
\{\langle\alpha,0,m\rangle,\langle\alpha,1,m\rangle\}\subseteq f[\{\langle\alpha,0,m\rangle,\langle\alpha,1,m\rangle\}],
\]
from which \eqref{sh24} follows.

By \eqref{sh24} and the surjectivity of $f$, we have $\kappa^+\subseteq f[\tilde{E}]$,
which is a contradiction since $|\tilde{E}|\leqslant\kappa$.
\end{proof}

\begin{lemma}\label{sh25}
In $\mathcal{V}$, the refinement postulate fails for surjective cardinals.
In particular, if $\mathfrak{a}=|A|$, $\mathfrak{b}=\kappa^+$, and $\mathfrak{c}_n=|A_n|$ for all $n\in\omega$,
then $\mathfrak{a}+\mathfrak{b}=^\ast\sum_{n\in\omega}\mathfrak{c}_n$ but there are no cardinals
$\mathfrak{a}_n,\mathfrak{b}_n$ {\upshape($n\in\omega$)} in $\mathcal{V}$ such that
\[
\mathfrak{a}=^\ast\sum_{n\in\omega}\mathfrak{a}_n,\quad\mathfrak{b}=^\ast\sum_{n\in\omega}\mathfrak{b}_n,
\quad\text{and}\quad\mathfrak{c}_n=^\ast\mathfrak{a}_n+\mathfrak{b}_n\text{ for all }n\in\omega.
\]
\end{lemma}
\begin{proof}
First, the function $f$ on $A$ defined by
\[
f(\langle\alpha,i,n\rangle)=
\begin{cases}
\alpha                     & \text{if $n=0$,}\\
\langle\alpha,i,n-1\rangle & \text{if $n>0$,}
\end{cases}
\]
is a surjection from $A$ onto $A\cup\kappa^+$.
Clearly, $f$ is fixed by every permutation in $\mathcal{G}$, so $f\in\mathcal{V}$.
Hence, in $\mathcal{V}$, $\mathfrak{a}+\mathfrak{b}=^\ast\sum_{n\in\omega}\mathfrak{c}_n$.

Assume to the contrary that there are cardinals
$\mathfrak{a}_n,\mathfrak{b}_n$ ($n\in\omega$) in $\mathcal{V}$ such that
\[
\mathfrak{a}=^\ast\sum_{n\in\omega}\mathfrak{a}_n,\quad\mathfrak{b}=^\ast\sum_{n\in\omega}\mathfrak{b}_n,
\quad\text{and}\quad\mathfrak{c}_n=^\ast\mathfrak{a}_n+\mathfrak{b}_n\text{ for all }n\in\omega.
\]
Since $\sum_{n\in\omega}\mathfrak{b}_n=^\ast\mathfrak{b}=\kappa^+$,
it follows that $\sum_{n\in\omega}\mathfrak{b}_n=\kappa^+$,
which implies that $\mathfrak{b}_m=\kappa^+$ for some $m\in\omega$. Hence,
\[
\mathfrak{c}_m=^\ast\mathfrak{a}_m+\mathfrak{b}_m=\mathfrak{a}_m+\kappa^+=\mathfrak{a}_m+\kappa^++\kappa^+=^\ast\mathfrak{c}_m+\kappa^+,
\]
contradicting Lemma~\ref{sh23}.
\end{proof}

Now, the next theorem immediately follows from Lemmas~\ref{sh22} and~\ref{sh25},
along with a transfer theorem of Pincus \cite[Theorem~4]{Pincus1977}.

\begin{theorem}
It is consistent with $\mathsf{ZF}+\mathsf{DC}_\kappa$ that surjective cardinals do not form a cardinal algebra.
\end{theorem}

\section{Concluding remarks}
To summarize, the article resolves the open question of whether surjective cardinals form a cardinal algebra,
and demonstrates that they indeed form a surjective cardinal algebra.
We conclude our article with some suggestions for further study.

In \cite{Tarski1949a}, Tarski gives a combinatorial proof, and in \cite{Schwartz2015},
Schwartz presents a game-theoretic proof of the Bernstein division theorem.
We wonder whether there are similar combinatorial or game-theoretic proofs
of the cancellation law for surjective cardinals (Corollary~\ref{sh26}).
We note that Tarski's combinatorial proof of the Bernstein division theorem relies
heavily on the refinement postulate for cardinals, suggesting that a combinatorial proof
of the cancellation law for surjective cardinals might be quite complex.

In \cite{Harrison2022}, Harrison-Trainor and Kulshreshtha give a complete axiomatization
of the logic of cardinality comparison without the axiom of choice.
It is worth replacing ``cardinality'' with ``surjective cardinality''
and exploring the corresponding complete axiomatization.

\end{document}